\documentclass[11pt]{amsart}
\usepackage{amsthm, amsmath, amssymb, hyperref, tikz}
\usepackage[initials]{amsrefs}

\newtheorem{theorem}{Theorem}[section]
\newtheorem{corollary}[theorem]{Corollary}
\newtheorem{lemma}[theorem]{Lemma}
\newtheorem{proposition}[theorem]{Proposition}

\theoremstyle{definition}

\newtheorem{definition}[theorem]{Definition}

\theoremstyle{remark}
\newtheorem{remark}[theorem]{Remark}

\newcommand{\F}{\mathbf{F}}
\newcommand{\K}{\mathbf{K}}
\newcommand{\G}{\mathbf{G}}

\newcommand{\A}{\mathcal{A}}
\newcommand{\forbG}{\mathrm{Forb}_{\mathrm{G}}}
\DeclareMathOperator{\bnd}{bnd}
\DeclareMathOperator{\forb}{Forb}
\DeclareMathOperator{\arity}{arity}
\DeclareMathOperator{\Wd}{Wd}
\DeclareMathOperator{\PW}{\mathbf{PW}}

\title{Indivisibility for Classes of Graphs}

\author[Guingona, Nusbaum, Padamsee, Parnes, Pippin, Zinman]{Vince Guingona, Felix Nusbaum, Zain Padamsee, Miriam Parnes, Christian Pippin, and Ava Zinman}

\thanks{This work was supported by Towson University and the National Science Foundation grant DMS-2149865.}

\begin{document}

\begin{abstract}
	We examine indivisibility for classes of graphs.  We show that the class of hereditarily $\alpha$-sparse graphs is indivisible if and only if $\alpha > 2$.  Additionally, we show that the following classes of graphs are indivisible: perfect graphs, cographs, and chordal graphs, and the following classes of graphs are not indivisible: threshold graphs, split graphs, and distance-hereditary graphs.
\end{abstract}

\maketitle


\section{Introduction}\label{Section_Introduction}

In this paper, we study indivisibility, a coloring property of classes of structures, which is related to the Ramsey property.  We consider this property on several classes of graphs, including hereditarily sparse graphs and graph classes that are characterized by forbidding certain induced subgraphs.

We say that a class of structures is \emph{indivisible} if, given any element $A$ of the class and any number of colors $k$, there exists another (ostensibly much larger) element of the class $B$ such that, no matter how we color $B$ with $k$ colors, there exists a monochromatic copy of $A$ in $B$.  This property is related to the Ramsey property -- in fact, if a class contains only one singleton structure up to isomorphism, then it is equivalent to the Ramsey property on singletons.  Indivisibility has been studied extensively, in particular by N.~Sauer; for example, see \cite{EZS89, EZS91, Sau14, Sau20}.  For example, it is known that the class of all finite simple graphs is indivisible, as is the class of all finite $K_n$-free graphs for $n \geq 3$ \cite{EZS89, Folk70, KR}.  In this paper, we examine indivisibility for other classes of graphs, including classes that do not have the amalgamation property and, hence, have no generic limit.

The first and fourth authors became interested in indivisibility via its relationship to configurations and the classification of theories by combinatorial complexity; see \cite{GP, GPS}.  Assuming that the index class of structures is indivisible generates enough ``uniformity'' to make understanding the configurations easier (without requiring the restrictive notion of being Ramsey).  Though the authors originally assumed that the index classes had the strong amalgamation property in \cite{GP}, this assumption was dropped in their later work with L.~Scow \cite{GPS}, so many of the classes presented in this paper work in the more general setting.

The motivation for studying the particular graph classes in this paper primarily comes from several papers in model theory.  One paper, by J.~Brody and M.~C.~Laskwoski \cite{BL}, examines model theoretic properties of the generic limit of the class of hereditarily $\alpha$-sparse graphs.  We borrow much of our terminology from that paper, including the function $\delta_\alpha$, and the classes $\K_\alpha$ and $\K^+_\alpha$.  In another paper \cite{LT}, M.~C.~Laskowski and C.~Terry study ``speeds'' of hereditary classes of $L$-structures in a finite relational language $L$.  That paper focuses on classes of structures with the hereditary property, since these classes have a well-defined boundary and, hence, have a classifying universal theory.  Restricting to classes of graphs, in \cite{LW}, A.~H.~Lachlan and R.~Woodrow classify the classes of graphs have the amalgamation property based on their boundaries (see Section \ref{Section_Amalgamation} for more details).

This paper consists of five further sections.  Section \ref{Section_Preliminatries} states some established definitions and previous results.  Section \ref{Section_SparseGraphs} examines hereditarily sparse graphs.  We prove that, for all positive real numbers $\alpha$, the class of hereditarily $\alpha$-sparse graphs is indivisible if and only if $\alpha > 2$ (Theorem \ref{Theorem_Sparseindivisible}) and the class of all strictly hereditarily $\alpha$-sparse graphs is indivisible if and only if $\alpha \geq 2$ (Theorem \ref{Theorem_Sparseindivisible2}).  Section \ref{Section_Forbidden} examines classes of graphs defined by forbidding induced subgraphs.  We prove that the following classes of graphs are indivisible: perfect graphs, cographs, and chordal graphs, and we prove that the following classes of graphs are not indivisible: threshold graphs, split graphs, and distance-hereditary graphs (Theorem \ref{Theorem_GraphIndivisible}).  In Section \ref{Section_Amalgamation}, we discuss indivisibility for classes of graphs with the amalgamation property.  Finally, Section \ref{Section_Conclusion} has some concluding remarks and open problems.


\section{Preliminaries}\label{Section_Preliminatries}

The objects considered in this paper are primarily simple graphs.  However, the paper is written from a model-theoretic perspective, so some of the notation may seem strange to combinatorialists.  For example, we will use $G$ to denote the graph $G$, but, abusing notation, we will also use it to denote the vertex set of $G$.  On the other hand, we use $E^G$ to denote the edge set of $G$.  Note that $E^G$ is a set of \emph{ordered} pairs that is irreflexive and symmetric (i.e., for all $a \in G$, $(a,a) \notin E^G$ and, for all $a, b \in G$, if $(a,b) \in E^G$, then $(b,a) \in E^G$).  So $|E^G|$ is twice the number of edges in $G$.  The important notion of ``substructure'' for us will be \emph{induced subgraph} instead of merely subgraph.

Let $L$ be a finite relational language.  For most of this paper, we will restrict to the language of graphs, $L = \{ E \}$, where $E$ is a binary relation symbol.  Let $\K$ be a class of finite $L$-structures that is closed under isomorphism.  Again, we will typically consider subclasses of $\G$, the class of all finite simple graphs.  For $A, B \in \K$, when we write $A \subseteq B$, we mean that $A$ is a substructure of $B$ (which, for graphs, means that $A$ is an induced subgraph of $B$).  An \emph{embedding} from $A$ to $B$ is an injective function $f : A \rightarrow B$ that preserves the interpretations all of the relation symbols in $L$.  For graphs, this means that $a$ and $b$ are adjacent if and only if $f(a)$ and $f(b)$ are adjacent.

We will mostly be interested in classes that have the hereditary property:

\begin{definition}[Hereditary Property]\label{Definition_HP}
	We say that $\K$ has the \emph{hereditary property} if, for all $A \in \K$ and $B \subseteq A$, $B \in \K$.
\end{definition}

Together with the hereditary property, the amalgamation property guarantees the existence of a generic limit structure (via Fra\"{i}ss\'{e}'s Theorem; see Theorem 7.1.2 of \cite{Hodges} or see \cite{Fr2000}).

\begin{definition}[Amalgamation Property]\label{Definition_AP}
	We say that $\K$ has the \emph{amalgamation property} if, for all $A, B_0, B_1 \in \K$ and for all embeddings $f_0 : A \rightarrow B_0$ and $f_1 : A \rightarrow B_1$, there exists $C \in \K$ and embeddings $g_0 : B_0 \rightarrow C$ and $g_1 : B_1 \rightarrow C$ such that $g_0 \circ f_0 = g_1 \circ f_1$.
\end{definition}

We say further that $\K$ has the \emph{strong amalgamation property} if, in the above setup, we require that $g_0(B_0) \cap g_1(B_1) = g_0(f_0(A))$.

For example, $\G$ itself has the (strong) amalgamation property, and the generic limit of $\G$ is the random graph (sometimes called the Rado graph).  However, it turns out that most classes of graphs do not have the amalgamation property, as noted in Corollary \ref{Corollary_AmalgamationProperty} below.

Throughout this paper, for any positive integer $n$, let $[n]$ denote the set of the first $n$ positive integers; i.e.,
\[
[n] = \{ 1, 2, \dots, n \}.
\]
For a positive integer $n$, let $K_n$ denote the complete graph on $n$ vertices, let $N_n$ denote the graph with $n$ vertices and no edges, let $P_n$ denote the path on $n$ vertices, and let $C_n$ denote the cycle on $n$ vertices (when $n \ge 3$).  We will assume that each of these graphs has universe $[n]$.

If $G$ is a graph, a \emph{clique} of $G$ is an induced subgraph of $G$ that is a complete graph and an \emph{independent set} of $G$ is an induced subgraph of $G$ that is a null graph.  For a graph $G$ and $a \in G$, let $\deg_G(a)$ denote the \emph{degree} of $a$ in $G$; that is,
\[
\deg_G(a) = \left| \left\{ b \in G : (a,b) \in E^G \right\} \right|.
\]
If $G$ is understood, we will drop references to it and write $\deg(a)$.  A graph $G$ is called \emph{$k$-regular} if, for all vertices $a \in G$, $\deg_G(a) = k$.

For us, graph colorings will be vertex colorings, and we are often interested in colorings that are not necessarily ``proper.''  A $k$-coloring of an $L$-structure $A$ is a function $c : A \rightarrow [k]$.  For $L$-structures $A$ and $B$ and a $k$-coloring $c$ of $B$, we say that $B$ has a \emph{monochromatic copy of $A$ with respect to $c$} if there exists an embedding $f : A \rightarrow B$ such that, for all $a_0, a_1 \in A$,
\[
c(f(a_0)) = c(f(a_1)).
\]
It will often be convenient to think of $f$ itself, rather than its image, as the ``copy'' of $A$ in $B$.

\begin{definition}[Indivisible]\label{Definition_Ind}
	We say that $\K$ is \emph{indivisible} if, for all $A \in \K$ and positive integers $k$, there exists $B \in \K$ such that, for every $k$-coloring $c$ of $B$, $B$ has a monochromatic copy of $A$ with respect to $c$.
\end{definition}

The notion of indivisibility for classes of structures is related to the notion of indivisibility for structures.  An $L$-structure $A$ is \emph{indivisible} if, for all positive integers $k$ and all $k$-colorings $c$ of $A$, there exists a monochromatic copy of $A$ in $A$ with respect to $c$.  It follows from Theorem 1 of \cite{EZS91} that, if $A$ is countable and indivisible as a structure, then its age is indivisible as a class of structures; for more on this, see Corollary 2.14 of \cite{GPS}.  If a class of structures is indivisible, then is its generic limit indivisible?  Not necessarily; see, for instance, Example 1 of \cite{Sau03} or Example 2.15 of \cite{GPS}.

Indivisibility is also related to the Ramsey Property.  Given $A \in \K$, we say that $\K$ has the \emph{$A$-Ramsey Property} if, for all $B \in \K$ and for all positive integers $k$, there exists $C \in \K$ such that, for any $k$-coloring $c$ of the $L$-embeddings of $A$ into $C$, there exists an $L$-embedding $g : B \rightarrow C$ and $i \in [k]$ such that, for all $L$-embeddings $f : A \rightarrow B$, $c(g \circ f) = i$.  When $A$ is a singleton structure and there is only one singleton structure in $\K$ up to isomorphism, then the $A$-Ramsey Property is equivalent to indivisibility (as $L$-embeddings of the only singleton structure into some $C \in \K$ can be identified with the elements of $C$).  Note that these two notions differ when there is more than one singleton structure up to isomorphism in $\K$.
If there exists a structure $A \in \K$ which has at least two different singleton substructures up to isomorphism, then $\K$ cannot be indivisible, since one can color any $B \in \K$ based on the isomorphism type of each singleton substructure; however, it may still have the Ramsey property for some singleton structure in $\K$.

As is the case with the Ramsey Property, it turns out that two colors suffice to show indivisibility.

\begin{theorem}[Theorem 1 of \cite{EZS91}]\label{Theorem_TwoColors}
	A class $\K$ is indivisible if and only if, for all $A \in \K$, there exists $B \in \K$ such that, for all $2$-colorings $c$ of $B$, $B$ has a monochromatic copy of $A$ with respect to $c$.
\end{theorem}

It is frequently helpful to describe a class of graphs by its boundary, which always exists and is unique up to isomorphism if the class has the hereditary property.

\begin{definition}[Boundary]\label{Definition_Bnd}
	Suppose that $\K$ has the hereditary property.  Then, a set $\A$ of $L$-structures is a \emph{boundary} of $\K$ if
	\begin{enumerate}
		\item for all $A, B \in \A$, there exists no embedding $f : A \rightarrow B$; and
		\item if $C$ is a finite $L$-structure, then $C \in \K$ if and only if for each $A \in \A$, there is no embedding $f : A \rightarrow C$.
	\end{enumerate}
\end{definition}

\begin{theorem}[\cite{Sau20}]
	If $\K$ has the hereditary property, then the boundary of $\K$ is unique up to isomorphism.
\end{theorem}

Since the boundary of $\K$ is unique when $\K$ has the hereditary property, we will give it a name, $\bnd(\K)$.  For example, $\bnd(\G)$ (the boundary of the class of all finite graphs) is the set containing a loop and a directed edge.
\vspace{-12pt}
\begin{center}
	\begin{tikzpicture}
		\draw (0,0) node[anchor = east] {$\bnd(\G) = \Bigl\{$};
		\draw[very thick, gray, ->] (0,0.1) .. controls (0,1) and (1,0) .. (0.1,0);
		\draw[very thick, gray, ->] (1.1,0) -- (1.9,0);
		\filldraw (0,0) circle (0.075);
		\draw (0.65,-0.15) node {$,$};
		\filldraw (1,0) circle (0.075);
		\filldraw (2,0) circle (0.075);
		\draw (2,0) node[anchor = west] {$\Bigr\}$};
	\end{tikzpicture}
\end{center}

Let $\A$ be a set of finite $L$-structures.  Define $\forb(\A)$ to be the class of all finite $L$-structures $B$ such that, for all $A \in \A$, there is no embedding of $A$ into $B$.  Here ``$\forb$'' stands for ``forbidden.''  Clearly $\forb(\A)$ has the hereditary property.  Furthermore, if we assume that, for all $A, B \in \A$, there exists no embedding $f : A \rightarrow B$, then, up to isomorphism,
\[
\bnd(\forb(\A)) = \A.
\]
On the other hand, for any class $\K$ with the hereditary property,
\[
\forb(\bnd(\K)) = \K.
\]

The next lemma follows from chasing definitions.

\begin{lemma}\label{Lemma_ForbiddenUnions}
	If $\{ \mathcal{A}_i : i \in I \}$ is a set of sets of finite $L$-structures, then
	\[
	\bigcap_{i \in I} \forb\left(\mathcal{A}_i\right) = \forb\left( \bigcup_{i \in I} \mathcal{A}_i \right).
	\]
\end{lemma}

If $\A$ a set of finite graphs, we use $\forbG(\A)$ to denote $\forb(\A \cup \bnd(\G))$.  That is, $\forbG(\A)$ is the class of all graphs forbidding graphs from $\A$ as induced subgraphs.

For $n \ge 3$, $\forbG(K_n)$ has the amalgamation property, and the generic limit of $\forbG(K_n)$ is usually called the $n$th \emph{Henson graph}; see \cite{Hen71}.  This generic limit is indivisible by \cite{EZS89}, which implies that $\forbG(K_n)$ is indivisible.  This was first shown directly in Theorem 2 of \cite{Folk70}.  More generally, Theorem 1 of \cite{NR} says that, if $\A$ is a finite set of $2$-connected graphs, then $\forbG(\A)$ is indivisible.  (Note that most such classes do not have the amalgamation property; see Corollary \ref{Corollary_AmalgamationProperty} below.)  For more examples, see \cite{RS, RSZ}.  This paper examines indivisibility for other classes of graphs, even those without the amalgamation property (hence without generic limits).

Let $L$ be a finite relational language and, for each $R \in L$, associate to it some positive real number $\alpha_R$.  Then, we can use this sequence $\overline{\alpha}$ to create a ``sparseness'' measure on $L$-structures as follows: For any $L$-structure $A$, let
\[
\delta_{\overline{\alpha}}(A) = |A| - \sum_{R \in L} \frac{\alpha_R}{\arity(R)!} \left|A^R\right|.
\]
(Here $\arity(R)$ denotes the arity of the relation symbol $R$, which for the edge relation $E$ on graphs is $2$.)  In particular, for graphs, we choose a single positive real number $\alpha$ (corresponding to $\alpha_E$) and we get, for any finite graph $G$,
\begin{equation}\label{Equation_DeltaAlpha}
	\delta_{\alpha}(G) = |G| - \alpha e(G),
\end{equation}
where $e(G)$ counts the number of edges in $G$; i.e., 
\[
e(G) = \frac{1}{2} \left| E^G \right|.
\]
In other words, $\delta_\alpha(G)$ is the difference between the number of vertices and $\alpha$ times the number of edges.  Note that, if $G$ is the disjoint union of graphs $A$ and $B$, then
\[
\delta_\alpha(G) = \delta_\alpha(A) + \delta_\alpha(B).
\]
With this, we can define two types of classes of sparse graphs.

\begin{definition}[$\K_\alpha$ and $\K_\alpha^+$]\label{Definition_Kalpha}
	Let $\alpha$ be a positive real number.  Define
	\begin{align*}
		\K_\alpha = & \ \left\{ A \in \G : \text{for all } B \subseteq A, \ \delta_\alpha(B) \ge 0 \right\} \\
		\K_\alpha^+ = & \ \left\{ A \in \G : \text{for all non-empty } B \subseteq A, \ \delta_\alpha(B) > 0 \right\}.
	\end{align*}
\end{definition}

\begin{remark}\label{Remark_Kalpha}
	For any positive real number $\alpha$ and $G \in \G$, $G \in \K_\alpha$ if and only if, for all subgraphs $H$ of $G$ with at least one edge,
	\[
	\frac{|H|}{e(H)} \geq \alpha.
	\]
	In other words, the ratio of vertices to edges is a bounded below by $\alpha$.  Therefore, $\K_\alpha$ is the class of all graphs that are hereditarily $\alpha$-sparse.  Note that we do not need to restrict to induced subgraphs here, as removing edges would only increase $|H|/e(H)$.
	
	This is also related to the notion of the maximum average degree of a graph.  By the Degree-Sum Formula, the average degree of a graph $G$ is $2e(G)/|G|$.  Therefore, we define the \emph{maximum average degree} of a graph $G$ to be the maximum of $2e(H)/|H|$ over all subgraphs $H$ of $G$.  Clearly, a graph $G$ is in $\K_\alpha$ if and only if it has maximum average degree at most $\frac{2}{\alpha}$.
\end{remark}

For any $L$-structure $A$, we can define the \emph{complement} of $A$, denoted $\overline{A}$, to be the $L$-structure with universe $A$ and, for all relation symbols $R \in L$, if $n = \arity(R)$, then $R^{\overline{A}}$ is the set of all $(a_1, \dots, a_n) \in A^n$ such that
\begin{itemize}
	\item $(a_1, \dots, a_n) \in R^A$ and there exist distinct $i, j \in [n]$ such that $a_i = a_j$; or
	\item $(a_1, \dots, a_n) \notin R^A$ and, for all distinct $i, j \in [n]$, $a_i \neq a_j$.
\end{itemize}
That is, we keep the same relation when there are repeated entries and swap the relation when each entry is distinct (including unary relation symbols).  Clearly $\overline{\overline{A}} = A$.  When applied to graphs, this is the usual graph complement and, when applied to strict linear orders, this is the reverse order.

The next lemma follows from definition-chasing.

\begin{lemma}\label{Lemma_EmbeddingComplement}
	Let $L$ be a finite relational language, let $A$ and $B$ be $L$-structures, and let $f : A \rightarrow B$ be an embedding.  Then, $f$ is an embedding of $\overline{A}$ into $\overline{B}$.
\end{lemma}

If $L$ is a finite relational language and $\K$ is a class of finite $L$-structures, we can define the complement class of $\K$ as follows:
\[
\overline{\K} = \left\{ \overline{A} : A \in \K \right\}.
\]
The following proposition follows immediately from the previous lemma.

\begin{proposition}\label{Proposition_PreserveComplement}
	Let $L$ be a finite relational language and let $\K$ be a class of finite $L$-structures.
	\begin{enumerate}
		\item $\overline{\K}$ has the hereditary property if and only if $\K$ has the hereditary property.
		\item $\overline{\K}$ has the amalgamation property if and only if $\K$ has the amalgamation property.
		\item $\overline{\K}$ is indivisible if and only if $\K$ is indivisible.
		\item for any set $\A$ of finite $L$-structures,
		\[
		\overline{\forb(\A)} = \forb(\overline{\A}).
		\]
	\end{enumerate}
\end{proposition}

In particular, note that, if $\A$ is a set of finite graphs, then
\begin{equation}\label{Equation_ForbComplement}
	\overline{\forbG(\A)} = \forbG(\overline{\A}),
\end{equation}
since $\overline{\bnd(\G)} = \bnd(\G)$.

Let $L$ be a finite relational language and let $A$ be an $L$-structure.  We say that $A$ is \emph{irreflexive} if, for all relation symbols $R \in L$ (say of arity $n$), for all $a_1, \dots, a_n \in A$, if $(a_1, \dots, a_n) \in R^A$, then $a_i \neq a_j$ for all distinct $i, j \in [n]$.

\begin{definition}[Lexicographic product]\label{Definition_LexProduct}
	Let $L$ be a finite relational langauge where each relation symbol is at least binary and let $A$ and $B$ be irreflexive $L$-structures.  The \emph{lexicographic product} of $A$ and $B$, denoted $A[B]$, is the $L$-structure with universe $A \times B$ and, for all relation symbols $R \in L$ say of arity $n$, define $R^{A[B]}$ by the set of all $((a_1, b_1), \dots, (a_n, b_n))$ such that $a_i \in A$ and $b_i \in B$ for all $i \in [n]$ and either
	\begin{enumerate}
		\item $(a_1, \dots, a_n) \in R^A$; or
		\item $a_1 = \dots = a_n$ and $(b_1, \dots, b_n) \in R^B$.
	\end{enumerate}
\end{definition}

This generalizes the notion of the lexicographic product of graphs and linear orders.  This is Definition 1.11 of \cite{Meir} and it is related to the lexicographic product discussed in \cite{GPS}, but missing the equivalence relation that identifies the first-coordinate, what N. Meir calls ``$s$.''

The following proposition follows from definition-chasing.

\begin{proposition}\label{Proposition_LexComplement}
	Let $L$ be a finite relational language where each relation symbol is at least binary and let $A$ and $B$ be two irreflexive $L$-structures.  Then,
	\[
	\overline{A[B]} = \overline{A} \left[ \overline{B} \right].
	\]
\end{proposition}


\section{Hereditarily sparse graphs}\label{Section_SparseGraphs}

In this section, we examine the indivisibility of hereditarily sparse graphs.  The primary tool we will use is the notion of ``gluing'' multiple copies of a single graph at a point.  This is something that can always be done in any indivisible class of finite structures in a fixed relational language.

\begin{theorem}\label{Theorem_IndivisibleStrongAmalgam}
	Let $L$ be a relational language and let $\K$ be a class of finite $L$-structures.  If $\K$ is indivisible, then, for all non-empty $A \in \K$ and for all natural numbers $n \geq 2$, there exist $B \in \K$, $b_0 \in B$, and embeddings $f_1, \dots, f_n : A \rightarrow B$ such that, for all distinct $i, j \in [n]$, $f_i(A) \cap f_j(A) = \{ b_0 \}$.
\end{theorem}

\begin{proof}
	Fix non-empty $A \in \K$ and $n \geq 2$.  Since $\K$ is indivisible, there exists $B \in \K$ such that, for all $n$-colorings of $B$, $B$ has a monochromatic copy of $A$.  Choose an $n$-coloring $c$ of $B$ such that the number of monochromatic copies of $A$ in $B$ with respect to $c$ is minimal (there is at least one by assumption).  Take any monochromatic copy of $A$ in $B$ with respect to $c$; i.e., $f : A \rightarrow B$ is an embedding of $A$ into $B$ such that $c(f(a)) = c(f(a'))$ for all $a, a' \in A$.  Fix $a_0 \in A$; let $b_0 = f(a_0)$ and $i_0 = c(b_0)$.  For each $i \in [n]$, define an $n$-coloring $c_i : B \rightarrow [n]$ by setting
	\[
	c_i(b) = \begin{cases} c(b) & \text{ if } b \neq b_0, \\ i & \text{ if } b = b_0 \end{cases}.
	\]
	Note that $c_{i_0} = c$ and let $f_{i_0} = f$.  For any $i \neq i_0$, note that $f$ is not a monochromatic copy of $A$ in $B$ with respect to $c_i$.  Thus, by choice of $c$ as having the minimal number of monochromatic copies of $A$, there must exist a monochromatic copy of $A$ in $B$ with respect to $c_i$ which is \emph{not} a monochromatic copy of $A$ in $B$ with respect to $c$; call it $f_i$.  Since $c$ and $c_i$ differ only on $b_0$, this implies that $b_0 \in f_i(A)$.  Therefore, $c_i(f_i(A)) = \{ i \}$.  For all distinct $i, j \in [n]$, since $c_i$ and $c_j$ agree on all points outside of $b_0$, yet the color of $c_i(f_i(A))$ and $c_j(f_j(A))$ disagree, we conclude that
	\[
	f_i(A) \cap f_j(A) = \{ b_0 \}.
	\]
\end{proof}

The preceding theorem is similar to Theorem 4.2 (i) of \cite{N05}, where here we replace ``Ramsey'' with ``indivisible'' and ``amalgamation'' with ``strong amalgamation over a singleton structure.''  It is also similar to Theorem 1.2 of \cite{DHKZ}, where they show that one can strong amalgamate a ``forest'' of copies of a fixed graph.

In the remainder of this section, we restrict to classes of graphs and use Theorem \ref{Theorem_IndivisibleStrongAmalgam} to examine the indivisibility of $\K_\alpha$ and $\K^+_\alpha$ for positive real numbers $\alpha$.
First, we list some connections between various classes of hereditarily sparse graphs.

\begin{lemma}\label{Lemma_CompareKalpha}
	For all $\alpha, \beta > 0$,
	\begin{enumerate}
		\item If $\alpha < \beta$, then $\K_\beta \subseteq \K^+_\alpha \subseteq \K_\alpha$.
		\item If $\alpha$ is irrational, then $\K_\alpha = \K^+_\alpha$.
		\item $\K^+_\alpha = \bigcup_{\gamma > \alpha} \K_\gamma$.
	\end{enumerate}
\end{lemma}

\begin{proof}
	(1): Follows from definitions.
	
	(2): This is the comment following Definition 2.2 of \cite{BL}.
	
	(3): Clearly (1) implies $\bigcup_{\gamma > \alpha} \K_\gamma \subseteq \K^+_\alpha$.  For the converse, fix $G \in \K^+_\alpha$ and let
	\[
	\beta = \min \left\{ |H|/e(H) : H \subseteq G, e(H) \geq 1 \right\}.
	\]
	One can easily verify that $\beta > \alpha$ and $G \in \K_\beta$.
\end{proof}

When studying hereditarily $\alpha$-sparse graphs, most papers are only concerned with $\alpha \in (0,1)$ (see, for example, \cite{BL}).  However, for completeness, we will consider all positive real numbers $\alpha$.  The classes $\K_\alpha$ and $\K^+_\alpha$ become somewhat uninteresting for $\alpha \ge 1$.

Let $\F$ be the class of all finite forests.  That is,
\[
\F = \forbG \left( \left\{ C_n : n \geq 3 \right\} \right).
\]
The next proposition follows from definitions.

\begin{proposition}\label{Proposition_classifyKlarge}
	For $\alpha \geq 1$, we classify $\K_\alpha$ and $\K^+_\alpha$ as follows:
	\begin{enumerate}
		\item $\K_1$ is the class of all finite graphs where each connected component has at most one cycle;
		\item $\K^+_1 = \F$;
		\item for $\alpha > 1$, $\K_\alpha$ is the class of all $G \in \F$ where each connected component of $G$ has at most $\left\lfloor\frac{1}{\alpha - 1} + 1 \right\rfloor$ vertices; and
		\item for $\alpha > 1$, $\K^+_\alpha$ is the class of all $G \in \F$ where each connected component of $G$ has at most $\left\lceil\frac{1}{\alpha - 1} \right\rceil$ vertices.
	\end{enumerate}
\end{proposition}

In particular, $\K^+_\alpha \subseteq \F$ for all $\alpha \geq 1$ and $\K_\alpha \subseteq \F$ for all $\alpha > 1$.  Moreover, note that $\K^+_\alpha$ is the class of all finite null graphs for $\alpha \geq 2$ and $\K_\alpha$ is the class of all finite null graphs for $\alpha > 2$.

For $\alpha \in (0,1)$, a classification of $\K_\alpha$ or $\K^+_\alpha$ becomes more difficult.  However, there are easy examples of graphs that belong in these classes.

\begin{lemma}\label{Lemma_RegularGraphs}
	For all $\alpha > 0$, for all positive integers $k \leq \frac{2}{\alpha}$, for all $k$-regular graphs $G$, $G \in \K_\alpha$.
\end{lemma}

\begin{proof}
	Take $H \subseteq G$ and $a \in H$.  Clearly $\deg_H(a) \leq \deg_G(a) = k$.  By the Degree-Sum Formula, $e(H) \leq k|H|/2$.  Thus,
	\[
	\delta_\alpha(H) = |H| - \alpha e(H) \geq |H| - \frac{k \alpha |H|}{2} \geq |H| - |H| = 0.
	\]
\end{proof}

This lemma allows us to specify when a complete graph is contained in $\K_\alpha$.

\begin{corollary}\label{Corollary_CompleteGraphsKalpha}
	Fix $\alpha > 0$.  Then, for all positive integers $n$,
	\[
	K_n \in \K_\alpha \text{ if and only if } n \leq \frac{2}{\alpha} + 1
	\]
	Hence, $K_m$ for $m = \left\lfloor \frac{2}{\alpha} + 1 \right\rfloor$ is the largest complete graph in $\K_\alpha$.
\end{corollary}

\begin{proof}
	If $n \leq \frac{2}{\alpha} + 1$, then $K_n$ is $(n-1)$-regular and $n-1 \leq \frac{2}{\alpha}$.  Therefore, by Lemma \ref{Lemma_RegularGraphs}, $K_n \in \K_\alpha$.  On the other hand, if $n > \frac{2}{\alpha}+1$, then
	\[
	e(K_n) = \binom{n}{2} = (n-1)\frac{n}{2} > \frac{2}{\alpha} \cdot \frac{n}{2} = \frac{n}{\alpha} = \frac{|K_n|}{\alpha}.
	\]
	Thus, $\delta_\alpha(K_n) = |K_n| - \alpha e(K_n) < 0$.  Therefore, $K_n \notin \K_\alpha$.
\end{proof}

We can generalize this to ``windmill graphs,'' which are created by taking disjoint copies of a complete graph and gluing them together at a single vertex.  More formally, the $(m,n)$-windmill graph is defined as follows.

\begin{definition}\label{Definition_Windmill}
	The \emph{$(m,n)$-windmill graph}, denoted $\Wd(m,n)$, is the graph whose vertex set is
	\[
	\Wd(m,n) = \left\{ (i,j) : i \in [n], j \in [m-1] \right\} \cup \{ 1 \}
	\]
	and whose edge set is
	\begin{align*}
		E^{\Wd(m,n)} = & \{ ((i,j_0), (i,j_1)) : i \in [n], j_0, j_1 \in [m-1], j_0 \neq j_1 \} \\ & \cup \left\{ ((i,j), 1) : i \in [n], j \in [m-1] \right\} \\ & \cup \left\{ (1, (i,j)) : i \in [n], j \in [m-1] \right\}.
	\end{align*}
	For any $i \in [n]$, call $\{ (i,j) : j \in [m-1] \} \cup \{ 1 \}$ the \emph{$i$th petal} of $\Wd(m,n)$.  We call $1$ the \emph{center} of $\Wd(m,n)$.  Note that $\Wd(m,n)$ has $n(m-1)+1$ vertices and $n \binom{m}{2}$ edges.
\end{definition}

For example, $\Wd(4,3)$ is the following graph.
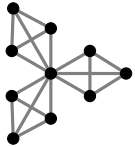
\begin{center}
	\begin{tikzpicture}
		\foreach \x in {0,120,240}{
			\draw[very thick, gray] (0,0) -- ({0.6*cos(\x+30)},{0.6*sin(\x+30)}) -- ({cos(\x)},{sin(\x)}) -- ({0.6*cos(\x-30)},{0.6*sin(\x-30)}) -- cycle;
			\draw[very thick, gray] (0,0) -- ({cos(\x)},{sin(\x)});
			\draw[very thick, gray] ({0.6*cos(\x+30)},{0.6*sin(\x+30)}) -- ({0.6*cos(\x-30)},{0.6*sin(\x-30)});
			\filldraw ({0.6*cos(\x+30)},{0.6*sin(\x+30)}) circle (0.075);
			\filldraw ({cos(\x)},{sin(\x)}) circle (0.075);
			\filldraw ({0.6*cos(\x-30)},{0.6*sin(\x-30)}) circle (0.075);
		}
		\filldraw (0,0) circle (0.075);
	\end{tikzpicture}
\end{center}

Which of these graphs belong to $\K_\alpha$?  This is answered by the following lemma.

\begin{lemma}\label{Lemma_WindmillGraphs}
	Fix $0 < \alpha \leq 2$ and let $m = \left\lfloor \frac{2}{\alpha} + 1 \right\rfloor$.  Then, for all positive integers $n$,
	\[
	\Wd(m,n) \in \K_\alpha \text{ if and only if } n \leq \frac{2}{(\alpha m - 2)(m - 1)}.
	\]
\end{lemma}

\begin{proof}
	From Corollary \ref{Corollary_CompleteGraphsKalpha}, we know that $K_m$ is the largest complete graph in $\K_\alpha$.  Note that, since $m > \frac{2}{\alpha}$, $\alpha m - 2 > 0$.  Moreover, since $\alpha \leq 2$, $m > 1$; thus, $m - 1 > 0$.  Since $m \leq \frac{2}{\alpha} + 1$,
	\[
	(\alpha m - 2)(m - 1) \leq (2 + \alpha - 2)\frac{2}{\alpha} = 2.
	\]
	Therefore,
	\[
	\frac{2}{(\alpha m - 2)(m - 1)} \geq 1.
	\]
	
	First, assume that $n \leq \frac{2}{(\alpha m - 2)(m - 1)}$.  This clearly implies
	\begin{align*}
		\delta_\alpha(\Wd(m,n)) & = n(m-1)+1 - \frac{ \alpha n m(m-1)}{2} \\ & = 1 + \frac{n(m-1)(2 - \alpha m)}{2} \geq 0.
	\end{align*}
	From this inequality, we obtain
	\[
	\frac{\alpha m}{2} \leq 1 + \frac{1}{n(m-1)}.
	\]
	Let $H \subseteq \Wd(m,n)$.  If the center of $\Wd(m,n)$ is not in $H$, then $H$ is a disjoint union of graphs of the form $K_j$ for $j < m$.  Hence, by Corollary \ref{Corollary_CompleteGraphsKalpha}, $\delta_\alpha(H) \geq 0$.  Therefore, we may assume that the center is in $H$.  For each $i \in [n]$, let $m_i$ denote the size of the intersection of the $i$th petal of $\Wd(m,n)$ with $H$, which is isomorphic to $K_{m_i}$.  As the center is in $H$, we have $1 \leq m_i \leq m$.  Moreover,
	\[
	\alpha \binom{m_i}{2} \leq \frac{\alpha m}{2} (m_i-1) \leq m_i-1 + \frac{m_i-1}{n(m-1)}.
	\]
	Thus,
	\[
	\delta_\alpha(H) = \sum_{i=1}^n \left(m_i - 1 - \alpha \binom{m_i}{2}\right) + 1 \geq -\frac{1}{n} \sum_{i=1}^n \frac{m_i-1}{m-1} + 1 \geq -\frac{n}{n} + 1 = 0.
	\]
	In either case, we get $\delta_\alpha(H) \geq 0$, so $\Wd(m,n) \in \K_\alpha$.
	
	Next, assume that $n > \frac{2}{(\alpha m - 2)(m - 1)}$.  Thus, $n(2 - \alpha m)(m-1) < -2$.
	Therefore,
	\begin{align*}
		\delta_\alpha(\Wd(m,n)) = & \ n(m-1)+1 - \frac{\alpha nm(m-1)}{2} = \\ & \ \frac{n(2-\alpha m)(m-1)}{2} + 1 < -1 + 1 = 0.
	\end{align*}
	Thus, $\Wd(m,n) \notin \K_\alpha$.
\end{proof}


We use these graphs to establish the following theorem.

\begin{theorem}\label{Theorem_Sparseindivisible}
	For all $\alpha > 0$, $\K_\alpha$ is indivisible if and only if $\alpha > 2$.
\end{theorem}

\begin{proof}
	As noted above, for $\alpha > 2$, $\K_\alpha$ is the class of all finite null graphs.  By the Pigeonhole principle, this is indivisible.
	
	Suppose $0 < \alpha \leq 2$.  Let
	\[
	m = \left\lfloor \frac{2}{\alpha} + 1 \right\rfloor \text{ and } n = \left\lfloor \frac{2}{(\alpha m - 2)(m - 1)} \right\rfloor.
	\]
	By Corollary \ref{Corollary_CompleteGraphsKalpha}, $m$ is maximal such that $K_m \in \K_\alpha$ and, by Lemma \ref{Lemma_WindmillGraphs}, $n$ is maximal such that $\Wd(m,n) \in \K_\alpha$.
	If $\K_\alpha$ is indivisible, then, by Theorem \ref{Theorem_IndivisibleStrongAmalgam} applied to $K_m$ and $n+1$, there exist $B \in \K_\alpha$, $b_0 \in B$, and embeddings $f_1, \dots, f_{n+1} : K_m \rightarrow B$ such that, for all distinct $i, j \in [n+1]$, $f_i(K_m) \cap f_j(K_m) = \{ b_0 \}$.  Thus, $\Wd(m,n+1)$ is a (not necessarily induced) subgraph of $B$.  This implies that $\Wd(m,n+1) \in \K_\alpha$, a contradiction.  Therefore, $\K_\alpha$ is not indivisible.
\end{proof}

We can prove a similar result for $\K^+_\alpha$, but this requires looking at a class of graphs we are calling \emph{pseudo-windmill graphs}.  These are graphs consisting of disjoint copies of a fixed graph glued together at a single vertex (not necessarily the same vertex in each copy).  Formally, we define pseudo-windmill graph class of $G$ and $n$ as follows.

\begin{definition}\label{Definition_PseudoWindmill}
	Let $G$ be a finite connected graph and $n \geq 1$.  For each choice of (not necessarily distinct) $a_1, \dots, a_n \in G$, define a graph $H$ whose vertex set is
	\[
	H = \{ (i, b) : i \in [n], b \in G \setminus \{ a_i \} \} \cup \{ 1 \}
	\]
	and whose edge set is
	\begin{align*}
		E^H = & \ \{ ((i,b), (i,c)) : i \in [n], b, c \in G \setminus \{ a_i \}, (b, c) \in E^G \} \\ & \cup \{ ((i,b), 1) : i \in [n], b \in G \setminus \{ a_i \}, (b, a_i) \in E^G \} \\ & \cup \{ (1, (i,b)) : i \in [n], b \in G \setminus \{ a_i \}, (a_i, b) \in E^G \}.
	\end{align*}
	For $i \in [n]$, call $\{ (i,b) : i \in [n], b \in G \setminus \{ a_i \} \} \cup \{ 1 \}$ the \emph{$i$th petal} of $H$ and call $1$ the \emph{center} of $H$.  The \emph{pseudo-windmill graph class} $\PW(G,n)$ is the class of all such graphs $H$ ranging over all choices of $a_1, \dots, a_n \in G$.
\end{definition}

Note that $\PW(K_m,n)$ contains only graphs isomorphic to $\Wd(m,n)$, hence this is a generalization of windmill graphs.  Moreover, $\PW(G,1)$ contains only graphs isomorphic to $G$.  Below are examples of two graphs in $\PW(P_3,6)$.
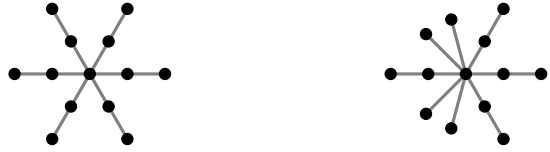
\begin{center}
	\begin{tikzpicture}
		\foreach \x in {0,60,120,180,240,300}{
			\draw[very thick, gray] (0,0) -- ({cos(\x)}, {sin(\x)});
			\filldraw ({0.5*cos(\x)}, {0.5*sin(\x)}) circle (0.075);
			\filldraw ({cos(\x)}, {sin(\x)}) circle (0.075);
		}
		\filldraw (0,0) circle (0.075);
		\foreach \x in {0,60,180,300}{
			\draw[very thick, gray] (5,0) -- ({5+cos(\x)}, {sin(\x)});
			\filldraw ({5+0.5*cos(\x)}, {0.5*sin(\x)}) circle (0.075);
			\filldraw ({5+cos(\x)}, {sin(\x)}) circle (0.075);
		}
		\foreach \x in {120,240}{
			\draw[very thick, gray] ({5+0.75*cos(\x-15)}, {0.75*sin(\x-15)}) -- (5,0) -- ({5+0.75*cos(\x+15)}, {0.75*sin(\x+15)});
			\filldraw ({5+0.75*cos(\x-15)}, {0.75*sin(\x-15)}) circle (0.075);
			\filldraw ({5+0.75*cos(\x+15)}, {0.75*sin(\x+15)}) circle (0.075);
		}
		\filldraw (5,0) circle (0.075);
	\end{tikzpicture}
\end{center}

To prove the analogous result of Theorem \ref{Theorem_Sparseindivisible} for $\K^+_\alpha$, we concentrate on the graph $G$ obtained by removing a single edge from $K_m$.  We now prove the analogue of Lemma \ref{Lemma_WindmillGraphs}.

\begin{lemma}\label{Lemma_PseudoWindmillSize}
	Fix $m \geq 4$, let $G$ be the graph $K_m$ with one edge removed, and let $\alpha = \frac{2}{m-1}$.  For all positive integers $n$,
	\begin{enumerate}
		\item If $n < \frac{m-1}{m-3}$, then $\PW(G,n) \subseteq \K^+_\alpha$.
		\item If $n \geq \frac{m-1}{m-3}$, then $\PW(G,n) \cap \K^+_\alpha = \emptyset$.
	\end{enumerate}
	(Note that, for $m \geq 6$, the first case only happens when $n = 1$.)
\end{lemma}

\begin{proof}
	First, since $m \geq 4$, $0 < \alpha \leq \frac{2}{3}$ and $1 < \frac{m-1}{m-3} \le 3$.  Note that, if $1 \leq \ell < m$, then
	\[
	\delta_\alpha(K_\ell) = \ell - \alpha \binom{\ell}{2} = \ell - \frac{\ell(\ell-1)}{m-1} \geq \ell - \frac{(m-1)(\ell-1)}{m-1} = 1.
	\]
	Similarly, if $1 \leq \ell \leq m$ and $B$ is $K_\ell$ with one edge removed, then
	\[
	\delta_\alpha(B) = \ell - \alpha \left( \binom{\ell}{2} - 1 \right) = \ell - \frac{\ell(\ell-1)}{m-1} + \frac{2}{m-1} \geq \frac{2}{m-1},
	\]
	with equality if $\ell = m$.
	
	(1): Assume $n < \frac{m-1}{m-3}$, fix $H \in \PW(G,n)$, and fix $A \subseteq H$ non-empty.  If $A$ does not contain the center of $H$, then $A$ is a disjoint union of graphs, each with at most $m - 1$ vertices, which are either complete or missing a single edge.  As noted above, each of these graphs have $\delta_\alpha > 0$.  Therefore, $\delta_\alpha(A) > 0$.  Thus, $H \in \K^+_\alpha$.
	
	Therefore, we may assume that $A$ contains the center.  For each $i \in [n]$, let $B_i$ denote the intersection of $A$ with the $i$th petal of $H$ and let $\ell_i = | B_i |$.  Then, $B_i$ is isomorphic to either $K_{\ell_i}$ and $1 \leq \ell_i < m$ or $K_{\ell_i}$ missing an edge and $1 \leq \ell_i \leq m$.  In the first case, $\delta_\alpha(B_i) - 1 \geq 0$ and, in the second case, $\delta_\alpha(B_i) - 1 \geq \frac{2}{m-1} - 1 = -\frac{m-3}{m-1}$.  Then,
	\[
	\delta_\alpha(A) = \sum_{i=1}^n \left( \delta_\alpha(B_i) - 1 \right) + 1 \geq 1 - n \cdot \frac{m-3}{m-1} > 1 - \frac{m-1}{m-3} \cdot \frac{m-3}{m-1} = 0.
	\]
	Hence, $H \in \K^+_\alpha$.
	
	(2): Assume $n \geq \frac{m-1}{m-3}$ and fix $H \in \PW(G,n)$.  Then,
	\[
	\delta_\alpha(H) = \sum_{i=1}^n \left( \delta_\alpha(G) - 1 \right) + 1 = 1 - n \cdot \frac{m-3}{m-1} \leq 1 - \frac{m-1}{m-3} \cdot \frac{m-3}{m-1} = 0.
	\]
	Hence, $H \notin \K^+_\alpha$.
\end{proof}

Finally, we get the desired result for $\K^+_\alpha$.

\begin{theorem}\label{Theorem_Sparseindivisible2}
	For all $\alpha > 0$, $\K^+_\alpha$ is indivisible if and only if $\alpha \geq 2$.
\end{theorem}

\begin{proof}
	As previously noted, for $\alpha \geq 2$, $\K^+_\alpha$ is the class of all finite null graphs, hence is indivisible.
	
	Suppose $0 < \alpha < 2$ and let $m = \left\lfloor \frac{2}{\alpha} + 1 \right\rfloor$.  If $m < \frac{2}{\alpha} + 1$, then, by the proofs of Corollary \ref{Corollary_CompleteGraphsKalpha} and Lemma \ref{Lemma_WindmillGraphs}, $K_m \in \K^+_\alpha$ and, for all positive integers $n$, $\Wd(m,n) \in \K^+_\alpha$ if and only if $n < \frac{2}{(\alpha m - 2)(m - 1)}$.  Therefore, by Theorem \ref{Theorem_IndivisibleStrongAmalgam}, $\K^+_\alpha$ is not indivisible.
	
	Thus, we may assume that $m = \frac{2}{\alpha} + 1$, and hence $\alpha = \frac{2}{m-1}$.  Since $\alpha < 2$, we may assume $m \geq 3$.  When $m = 3$, $\alpha = 1$, and by Proposition \ref{Proposition_classifyKlarge} (2), $\K^+_1 = \F$, the class of all finite forests.  This is not indivisible; for example, let $A = K_2$, take any $B \in \F$, and let $c$ be a proper $2$-coloring of $B$ (which exists since $B$ is bipartite).  Then, there exists no monochromatic copy of $A$ in $B$.
	
	Therefore, we may assume that $m \geq 4$.  Let $G$ be $K_m$ with one edge removed and let
	\[
	n = \left\lceil \frac{m-1}{m-3} \right\rceil.
	\]
	By Lemma \ref{Lemma_PseudoWindmillSize}, $\PW(G,n) \cap \K^+_\alpha = \emptyset$ (but $G \in \K^+_\alpha$).
	If $\K^+_\alpha$ is indivisible, then, by Theorem \ref{Theorem_IndivisibleStrongAmalgam}, there exist $B \in \K^+_\alpha$, $b_0 \in B$, and embeddings $f_1, \dots, f_n : G \rightarrow B$ such that, for all distinct $i, j \in [n]$, $f_i(G) \cap f_j(G) = \{ b_0 \}$.  That is, $B$ contains a (not necessarily induced) subgraph that is isomorphic to one in $\PW(G,n)$, a contradiction.  Therefore, $\K^+_\alpha$ is not indivisible.
\end{proof}


\section{Classes of graphs with forbidden substructures}\label{Section_Forbidden}

Before we examine graphs in particular, we begin by looking at classes of arbitrary $L$-structures for some relational language $L$ where each relation symbol is at least binary (this covers, for example, classes of graphs or classes of linear orders).

First, we consider the following lemma, whose proof is straightforward.

\begin{lemma}\label{Lemma_Embeddings}
	Let $L$ be a relational language where each relation symbol is at least binary.  For all irreflexive $L$-structures $A$ and $B$, for all $a^* \in A$ and functions $h : A \rightarrow B$, the functions $f : A \rightarrow A[B]$ and $g : B \rightarrow A[B]$ given by $f(a) = (a, h(a))$ and $g(b) = (a^*,b)$ are embeddings.
\end{lemma}

This gives rise to a sufficient condition for indivisibility, which we will later use to examine classes of graphs.

\begin{theorem}\label{Theorem_Miriams}
	Let $L$ be a relational language where each relation symbol is at least binary and let $\K$ be a class of irreflexive $L$-structures (note that the language nor the structures need to be finite).  If, for all $A \in \K$, $A[A] \in \K$, then $\K$ is indivisible.
\end{theorem}

\begin{proof}
	By Theorem \ref{Theorem_TwoColors}, it suffices to check indivisibility for $2$-colorings.  Fix $A \in \K$ and let $B = A[A]$.  Fix any $2$-coloring $c$ of $B$.  Suppose that, for some $a^* \in A$, $c( \{ a^* \} \times A ) = \{ 1 \}$.  Let $f : A \rightarrow B$ be given by, for all $a \in A$, $f(a) = (a^*, a)$.  By Lemma \ref{Lemma_Embeddings}, $f$ is an embedding of $A$ into $B$.  Thus, $f$ is a monochromatic copy of $A$ in $B$.
	
	Thus, we may assume that, for each $a \in A$, there exists $b_a \in A$ such that $c(a, b_a) = 2$.  Let $g : A \rightarrow B$ be given by $g(a) = (a,b_a)$.  By Lemma \ref{Lemma_Embeddings}, $g$ is an embedding of $A$ into $B$.  Thus, $g$ is a monochromatic copy of $A$ in $B$.
\end{proof}

Note that this also works for classes of infinite structures.  For example, we get the following corollary.

\begin{corollary}\label{Corollary_Graphs}
	Let $\kappa$ be an infinite cardinal and let $\G_{< \kappa}$ be the class of all graphs of size less than $\kappa$.  Then, $\G_{< \kappa}$ is indivisible.
\end{corollary}

\begin{proof}
	If $A$ is a graph of size less than $\kappa$, then $A[A]$ is a graph of size less than $\kappa$.  Thus, by Theorem \ref{Theorem_Miriams}, $\G_{< \kappa}$ is indivisible.
\end{proof}

For the remainder of this section, we will let $L$ be the language of graphs.  We will consider $\forbG(\A)$ for $\A$ a set of finite graphs.  We begin with a lemma that shows that graph classes that only forbid sufficiently large paths or cycles are closed under lexicographic product.

\begin{lemma}\label{Lemma_ForbidPathCycle}
	If $F$ is either $P_n$ for $n \geq 4$ or $C_n$ for $n \geq 5$, and $A, B \in \forbG(F)$, then $A[B] \in \forbG(F)$.
\end{lemma}

\begin{proof}
	Fix distinct $(a_1, b_1), \dots, (a_n, b_n) \in A[B]$.  Suppose that $H$, the induced subgraph of $A[B]$ on these vertices, is isomorphic to $F$.  On the one hand, if $a_1 = \dots = a_n$, then the induced subgraph of $B$ on $\{ b_1, \dots, b_n \}$ is isomorphic to $F$, contrary to the fact that $B \in \forbG(F)$.  On the other hand, if the $a_i$'s are distinct, then the induced subgraph of $A$ on $\{ a_1, \dots, a_n \}$ is isomorphic to $F$, contrary to the fact that $A \in \forbG(F)$.  So neither of these cases hold.
	
	Fix distinct $i_0, i_1 \in [n]$ such that $a_{i_0} = a_{i_1}$.  Since $F$ is connected, there exists $j_0 \in [n]$ such that $a_{i_0} \neq a_{j_0}$ and $(a_{i_0}, a_{j_0}) \in E^A$.
	
	If $(b_{i_0}, b_{i_1}) \in E^B$, then $\{ (a_{i_0}, b_{i_0}), (a_{i_1}, b_{i_1}), (a_{j_0}, b_{j_0}) \}$ form a clique in $A[B]$, contrary to the fact that $H \cong F$.
	
	If there exists $i \in [n]$ with $i \neq i_0$, $i \neq i_1$, and $a_i = a_{i_0}$, then the degree of $(a_{j_0}, b_{j_0})$ in $H$ is at least three, contrary to the fact that $H \cong F$.  If there exists $j \in [n]$ with $j \neq j_0$ and $(a_{i_0}, a_j) \in E^A$, then the induced subgraph of $A[B]$ on
	\[
	\{ (a_{i_0}, b_{i_0}), (a_{j_0}, b_{j_0}), (a_{i_1}, b_{i_1}), (a_j, b_j) \}
	\]
	contains $C_4$ as a subgraph, contrary to the fact that $H \cong F$.
	
	Therefore, the degrees of $(a_{i_0}, b_{i_0})$ and $(a_{i_1}, b_{i_1})$ in $H$ are each exactly $1$.  Clearly this is a contradiction if $F = C_n$.  Moreover, if $F = P_n$ with $n \geq 4$, then this is also a contradiction, since both $(a_{i_0}, b_{i_0})$ and $(a_{i_1}, b_{i_1})$ are adjacent to a common vertex, $(a_{j_0}, b_{j_0})$.
\end{proof}

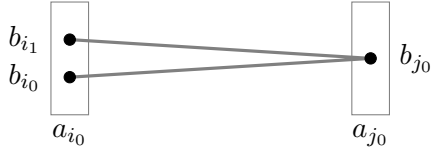
\begin{center}
	\begin{tikzpicture}
		\draw[gray] (0,0) rectangle (0.5,1.5);
		\draw[gray] (4,0) rectangle (4.5,1.5);
		\draw (0.25,0) node[anchor = north] {$a_{i_0}$};
		\draw (4.25,0) node[anchor = north] {$a_{j_0}$};
		\draw[very thick, gray] (0.25,0.5) -- (4.25,0.75) -- (0.25,1);
		\filldraw (0.25,0.5) circle (0.075);
		\draw (0,0.5) node[anchor = east] {$b_{i_0}$};
		\filldraw (0.25,1) circle (0.075);
		\draw (0,1) node[anchor = east] {$b_{i_1}$};
		\filldraw (4.25,0.75) circle (0.075);
		\draw (4.5,0.75) node[anchor = west] {$b_{j_0}$};
	\end{tikzpicture}
\end{center}

The next lemma follows immediately from equation \eqref{Equation_ForbComplement} and Proposition \ref{Proposition_LexComplement}.

\begin{lemma}\label{Lemma_ForbidComplement}
	Let $F$ be a finite graph and suppose that, for all $A, B \in \forbG(F)$, $A[B] \in \forbG(F)$.  Then, for all $A, B \in \forbG(\overline{F})$, $A[B] \in \forbG(\overline{F})$.
\end{lemma}

We put all of this together to get the following proposition about the indivisibility of graphs which forbid certain substructures.

\begin{proposition}\label{Proposition_ForbidF}
	Let
	\[
	\mathcal{A} \subseteq \{ C_n : n \geq 5 \} \cup \{ P_n : n \geq 4 \} \cup \left\{ \overline{C_n} : n \geq 5 \right\} \cup \left\{ \overline{P_n} : n \geq 4 \right\}.
	\]
	Then, $\forbG(\mathcal{A})$ is indivisible.
\end{proposition}

\begin{proof}
	By Lemma \ref{Lemma_ForbidPathCycle} and Lemma \ref{Lemma_ForbidComplement}, if $F$ is either $C_n$ or $\overline{C_n}$ for $n \geq 5$ or $P_n$ or $\overline{P_n}$ for $n \geq 4$, then $\forbG(F)$ is closed under lexicographic product.  By Lemma \ref{Lemma_ForbiddenUnions}, $\forbG(\mathcal{A})$ is closed under lexicographic product.  By Theorem \ref{Theorem_Miriams}, $\forbG(\mathcal{A})$ is indivisible.
\end{proof}

Note that the set $\A$ in this proposition may be infinite and may also include some graphs that are not $2$-connected, in contrast to Theorem 1 of \cite{NR}.

This proposition has a number of consequences, as many interesting classes of graphs are of the form $\forbG(\mathcal{A})$ for some such $\mathcal{A}$.  For example, consider ``cographs'' (e.g., Exercise 8.1.3 of \cite{West}).

\begin{definition}\label{Definition_Cograph}
	A finite graph $G$ is called a \emph{complement-reducible graph} (or \emph{cograph}) if it forbids $P_4$ as an induced subgraph.
\end{definition}

The following Corollary first appears as Theorem 4.2 in \cite{RS}.

\begin{corollary}\label{Corollary_Cographs}
	The class of all finite cographs, $\forbG(P_4)$, is indivisible.
\end{corollary}

Next, consider ``perfect graphs'' (e.g., Definition 5.3.18 of \cite{West}).

\begin{definition}\label{Definition_PerfectGraph}
	A finite graph $G$ is called a \emph{perfect graph} if, for every induced subgraph $H$ of $G$, the chromatic number of $H$ is equal to the size of largest clique of $H$.
\end{definition}

Cographs are a subclass of perfect graphs, as can be observed by the following forbidden induced subgraph characterization of perfect graphs.

\begin{theorem}[Theorem 1.2 of \cite{CRST}]\label{Theorem_PerfectClassify}
	A finite graph $G$ is perfect if and only if $G$ forbids $C_n$ and $\overline{C_n}$ for all odd $n \geq 5$.  That is, the class of all perfect graphs is
	\[
	\forbG\left( \left\{ C_n : n \geq 5 \text{ odd} \right\} \cup \left\{ \overline{C_n} : n \geq 5 \text{ odd} \right\} \right).
	\]
\end{theorem}

Therefore, the following corollary holds.

\begin{corollary}\label{Corollary_PerfectGraphs}
	The class of all finite perfect graphs is indivisible.
\end{corollary}

Note that this is also a consequence of \cite{RP}, which proves that, for all finite graphs $A$ and $B$, $A[B]$ is perfect if and only if $A$ and $B$ are perfect.

Next, consider the class of ``chordal graphs'' (e.g., Definition 5.3.15 of \cite{West}), which is also a subclass of the perfect graphs.

\begin{definition}\label{Definition_Chordal}
	A finite graph $G$ is called a \emph{chordal graph} if it forbids $C_n$ as an induced subgraph for all $n \geq 4$.
\end{definition}

Clearly Proposition \ref{Proposition_ForbidF} does not directly apply to the class of chordal graphs, since $C_4$-free graphs are not closed under lexicographic product (for example, $K_2[N_2]$ is isomorphic to $C_4$).  Therefore, we need a slightly different argument to prove the following proposition.

\begin{proposition}\label{Proposition_Chordal}
	The class of all finite chordal graphs is indivisible.
\end{proposition}

We employ a characterization of chordal in terms of a ``simplicial elimination ordering.''

\begin{theorem}[Theorem 5.3.17 of \cite{West}]\label{Theorem_ChordalClassify}
	A finite graph $G$ is chordal if and only if there is an ordering of the vertices of $G$, $\{ a_1, a_2, \dots, a_n \}$, such that, for all $i \in [n]$, the neighbors of $a_i$ in the graph induced by $\{ a_1, \dots, a_i \}$ form a complete graph.
\end{theorem}

In other words, $G$ can be constructed by adding one vertex at a time, ensuring that each new vertex is adjacent to a clique.

Next, consider the following lemma about creating a new chordal graph from two existing chordal graphs by making the entirety of one graph adjacent to a clique in the other.

\begin{lemma}\label{Lemma_AttachingGraphs}
	If $A$ and $B$ are finite chordal graphs and $C$ is a clique in $A$, then the following graph $G$ is chordal: $G$ has vertex set $A \sqcup B$, $E^G \cap A^2 = E^A$, $E^G \cap B^2 = E^B$, and, for each $a \in A$ and $b \in B$, $(a,b), (b,a) \in E^G$ if and only if $a \in C$.
\end{lemma}

\begin{proof}
	Suppose, towards a contradiction, $G$ contains $H \cong C_n$ as an induced subgraph for some $n \ge 4$.  Since $A$ and $B$ are chordal, $H \not\subseteq A$ and $H \not\subseteq B$.  Since $H$ is connected, there exists $a \in H \cap A$ and $b \in H \cap B$ that are adjacent.  Then, $a \in C$ and $a$ is adjacent to every vertex in $B$.  If $|H \cap B| > 2$, then the degree of $a$ in $H$ is at least $3$, contrary to $H \cong C_n$.  Therefore, $|H \cap B| \leq 2$.  Since $|H| \geq 4$, we have $|H \cap A| \geq 2$.  Since $H \setminus \{ a \} \cong P_{n-1}$, it is still connected, hence there exists $a' \in (H \setminus \{ a \}) \cap A$ adjacent to $b$.  However, this means $a' \in C$, hence $a'$ is adjacent to $a$.  Then $\{ a, a', b \}$ form a clique in $H$, a contradiction.
\end{proof}

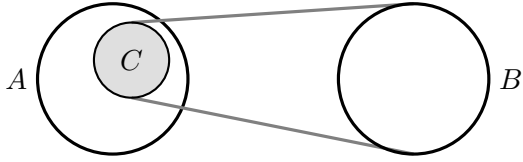
\begin{center}
	\begin{tikzpicture}
		\draw[very thick] (0,0) circle (1);
		\draw (-1,0) node[anchor = east] {$A$};
		\filldraw[gray!25!white] (0.25,0.25) circle (0.5);
		\draw[very thick, gray] (0.25,0.75) -- (4,1);
		\draw[very thick, gray] (0.25,-0.25) -- (4,-1);
		\draw[thick] (0.25,0.25) node {$C$} circle (0.5);
		\draw[very thick] (4,0) circle (1);
		\draw (5,0) node[anchor = west] {$B$};
	\end{tikzpicture}
\end{center}

We use Theorem \ref{Theorem_ChordalClassify} and Lemma \ref{Lemma_AttachingGraphs} to prove Proposition \ref{Proposition_Chordal}.
Note that this proof is a variant of an argument from folklore.  For example, it is similar to the proof that the Rado graph is indivisible; see \cite{EZS89, Fr2000}.

\begin{proof}[Proof of Proposition \ref{Proposition_Chordal}]
	By induction on the construction of $A$ as in Theorem \ref{Theorem_ChordalClassify}.  If $A = \emptyset$, then let $B = \emptyset$.  If $A$ is a single vertex, let $B$ be a single vertex.
	
	Suppose now that $A$ is constructed by creating a new vertex $v$ and making it adjacent to a clique $C$ in some chordal graph $A'$ with $|A'| \geq 1$.  By induction, there exists a chordal graph $B'$ such that, for all $2$-colorings $c'$ of $B'$, $B'$ has a monochromatic copy of $A'$ with respect to $c'$.
	
	Let $k = |C|$ and let $C_1, \dots, C_m$ enumerate all of the cliques in $B'$ of size $k$ (distinct, but not necessarily disjoint).  Then, form a new graph $B$ with vertex set
	\[
	B' \sqcup \{ (i, a) : i \in [m], a \in A \}.
	\]
	The edge set of $B$ will be given by:
	\begin{itemize}
		\item $E^B \cap (B')^2 = E^{B'}$;
		\item $((i,a), (i',a')) \in E^B$ if and only if $i = i'$ and $(a, a') \in A$; and
		\item $((i, a), b) \in E^B$ and $(b, (i,a)) \in E^B$ if and only if $b \in C_i$.
	\end{itemize}
	In other words, $B$ is $B'$ together with, for each $k$-clique $C'$ in $B'$, a copy of $A$ all of whose vertices are adjancent to each vertex of $C'$.  By Lemma \ref{Lemma_AttachingGraphs} and induction, $B$ is chordal.  (Note that, if $k = 0$, then there is exactly one clique of $B'$ of size $0$, namely $C_1 = \emptyset$.  Therefore, $B$ is a disjoint union of $B'$ and $A$.)  We show that this $B$ works.
	
	Let $c$ be a $2$-coloring of $B$ (which restricts to a $2$-coloring of $B'$).  By assumption, there exists a monochromatic copy $f$ of $A'$ in $B'$ with respect to $c$; without loss of generality, suppose that $c(f(A')) = \{ 1 \}$ (i.e., this copy is colored $1$).  Then, $f(C) = C_i$ for some $i$.  If $c((i,a)) = 2$ for all $a \in A$, then this would be a monochromatic copy of $A$ in $B$ with respect to $c$.  Otherwise, $c((i,a)) = 1$ for some $a \in A$, in which case $f(A') \cup \{ (i,a) \}$ is a monochromatic copy of $A$ in $B$ with respect to $c$.
\end{proof}

Not all graph classes that are characterized by forbidding induced subgraphs are indivisible, not even all subclasses of perfect graphs.  For example, we look at ``threshold graphs'' (e.g., \cite{CH}).

\begin{definition}\label{Definition_Threshold}
	A \emph{threshold graph} is a graph that is recursively defined as follows:
	\begin{itemize}
		\item The empty graph is a threshold graph;
		\item if $G$ is a threshold graph, then the addition of a single new isolated vertex to $G$ is a threshold graph; and
		\item if $G$ is a threshold graph, then the addition of a single new vertex adjacent to every vertex in $G$ is a threshold graph.
	\end{itemize}
\end{definition}

\begin{theorem}[Theorem 1 of \cite{CH}]\label{Theorem_ThresholdClassify}
	A finite graph $G$ is a threshold graph if and only if $G$ forbids $P_4$, $C_4$, and $\overline{C_4}$ as induced subgraphs.
\end{theorem}

As we alluded to earlier, it turns out that the class of threshold graphs is not indivisible.

\begin{proposition}\label{Proposition_Threshold}
	The class of all finite threshold graphs is not indivisible.
\end{proposition}

\begin{proof}
	Let $A = \overline{P_3}$, which is clearly a threshold graph.
	\begin{center}
		\begin{tikzpicture}
			\draw[very thick, gray] (0,0) -- (1,0);
			\foreach \x in {0,1,2}
			\filldraw (\x,0) circle (0.075);
			\draw (-0.5,0) node {$A$};
		\end{tikzpicture}
	\end{center}
	For each non-empty threshold graph $B$, we recursively define a $2$-coloring $c$ of $B$ using the construction given in Definition \ref{Definition_Threshold} as follows:
	\begin{itemize}
		\item If $B$ is constructed from a threshold graph $B'$ by adding an additional isolated vertex $v$ and $c'$ is the $2$-coloring on $B'$, define $c : B \rightarrow [2]$ extending $c'$ given by setting $c(v) = 1$.
		\item If $B$ is constructed from a threshold graph $B'$ by adding an additional vertex $v$ adjacent to all vertices in $B'$ and $c'$ is the $2$-coloring of $B'$, define $c: B \rightarrow [2]$ extending $c'$ by setting $c(v) = 2$.
	\end{itemize}
	From the above construction, it is clear that, for any threshold graph $B$, no two vertices of $B$ that are colored $1$ can be adjacent.  Therefore, $c^{-1}(\{ 1 \})$ does not contain a copy of $A$.  On the other hand, any two vertices of $B$ that are colored $2$ are adjacent.  Therefore, $c^{-1}(\{ 2 \})$ does not contain a copy of $A$.  Hence, there is no monochromatic copy of $A$ in $B$ with respect to $c$.
\end{proof}

Next, consider ``split graphs'' (e.g., Exercise 8.1.17 of \cite{West}).

\begin{definition}\label{Definition_Split}
	A finite graph $G$ is called a \emph{split graph} if its vertices can be partitioned into two sets, $A$ and $B$, where $A$ is a clique and $B$ is an independent set.
\end{definition}

\begin{theorem}[\cite{HS}]\label{Theorem_SplitClassify}
	A finite graph $G$ is a split graph if and only if $G$ forbids $C_4$, $C_5$, and $\overline{C_4}$ as induced subgraphs.
\end{theorem}

It turns out that the class of split graphs is also not indivisible.

\begin{proposition}\label{Proposition_Split}
	The class of all finite split graphs is not indivisible.
\end{proposition}

\begin{proof}
	Let $A = P_3$, which is clearly a split graph.  Fix a split graph $B$.  Suppose that $B = C \sqcup D$, where $C$ is a clique and $D$ is a null graph.  Then, let $c : B \rightarrow [2]$ be given by
	\[
	c(b) = \begin{cases} 1 & \text{ if } b \in C, \\ 2 & \text{ if } b \in D \end{cases}.
	\]
	Then there exists no monochromatic copy of $A$ in $B$ with respect to $c$, as $A$ cannot be embedded in either a complete graph or a null graph.
\end{proof}

Another subclass of the perfect graphs with a more complicated forbidden induced subgraph characterization is the class of ``distance-hereditary graphs'' \cite{How}.

\begin{definition}[\cite{How}]\label{Definition_Distance}
	A finite graph $G$ is a \emph{distance-hereditary graph} if, for all connected induced subgraphs $H$ of $G$, for all $a, b \in H$, the distance between $a$ and $b$ in $H$ is equal to the distance between $a$ and $b$ in $G$.
\end{definition}

\begin{theorem}[\cite{How}]\label{Theorem_DistanceForbid}
	A finite graph $G$ is a distance-hereditary graph if and only if $G$ forbids $\overline{P_5}$, $C_n$ for all $n \geq 5$, and the following two graphs:
	
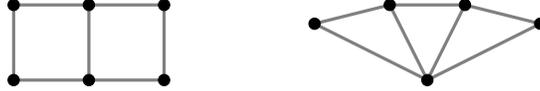
\begin{center}
		\begin{tikzpicture}
			\draw[very thick, gray] (0,0) rectangle (2,1);
			\draw[very thick, gray] (1,0) -- (1,1);
			\foreach \x in {(0,0),(1,0),(2,0),(0,1),(1,1),(2,1)}
			\filldraw \x circle (0.075);
			\draw[very thick, gray] (5.5,0) -- (4,0.75) -- (5,1) -- (6,1) -- (7,0.75) -- cycle;
			\draw[very thick, gray] (6,1) -- (5.5,0) -- (5,1);
			\foreach \x in {(4,0.75),(5,1),(6,1),(7,0.75),(5.5,0)}
			\filldraw \x circle (0.075);
		\end{tikzpicture}
	\end{center}
\end{theorem}

However, we will make use of recursive characterization of the class of distance-hereditary graphs.

\begin{theorem}[Theorem 1 of \cite{BM}]\label{Theorem_DistanceClassify}
	A finite graph $G$ is a distance-hereditary graph if and only if $G$ can be recursively constructed, starting with a single vertex, by the following three operations:
	\begin{itemize}
		\item add a single vertex adjacent to exactly one vertex in $G$;
		\item replace any one vertex in $G$ with $N_2$; and
		\item replace any one vertex in $G$ with $K_2$.
	\end{itemize}
	(The second two operations are called ``splitting a vertex'' via ``false twins'' and ``true twins,'' respectively.)
\end{theorem}

Using this construction will allow us to show that the class of all distance-hereditary graphs is not indivisible.

\begin{proposition}\label{Proposition_Distance}
	The class of all finite distance-hereditary graphs is not indivisible.
\end{proposition}

\begin{proof}
	Let $A = P_4$, which is a distance-hereditary graph.  For any distance-hereditary graph $B$, we use the construction in Theorem \ref{Theorem_DistanceClassify} to recursively define a $2$-coloring of $B$.  Suppose $B'$ is a distance-hereditary graph and $c'$ is the $2$-coloring of $B'$ we have constructed so far.  We define $c : B \rightarrow [2]$ extending $c'$ as follows: If $B$ is created by adding a vertex $v$ adjacent to a single vertex $w \in B'$, then define $c(v)$ to be different from $c'(w)$.  If $B$ is created by duplicating a vertex $w \in B'$ (by either replacing it with $K_2$ or $N_2$), then $c$ assigns both new vertices the original color, $c'(w)$.
	
	Toward a contradiction, suppose that there is a monochromatic copy of $A$ in $B$ with respect to $c$.  Halt the construction when the first monochromatic copy of $A$ in $B$ with respect to $c$ appears; let $f$ be that copy and let $H = f(A)$.  So $H \cong A = P_4$.  This construction will not halt when we add a single vertex adjacent to one other vertex, as those vertices will have different colors with respect to $c$.  Therefore, it must halt when we split a vertex; call these new vertices $a_1$ and $a_2$.  Clearly both $a_1, a_2 \in H$, as otherwise we would have had a monochromatic copy of $A$ earlier in the construction.  By construction $a_1$ and $a_2$ are adjacent to the same vertices in $B$.  Let $b_1, b_2$ be the other two vertices in $H$.  Since $H$ is connected, we must have that $a_1$ and $a_2$ are each adjacent to both $b_1$ and $b_2$.  Therefore, $H$ contains $C_4$ as a subgraph, contrary to the fact that $H \cong P_4$.
	
	Thus, there is no monochromatic copy of $A$ in $B$ with respect to $c$.  So the class of all finite distance-hereditary graphs is not indivisible.
\end{proof}

To conclude this section, we summarize all of the results in the following theorem and table.

\begin{theorem}\label{Theorem_GraphIndivisible}
	The following classes of finite graphs are indivisible:
	\begin{itemize}
		\item cographs;
		\item perfect graphs;
		\item chordal graphs.
	\end{itemize}
	The following classes of finite graphs are not indivisible:
	\begin{itemize}
		\item threshold graphs;
		\item split graphs;
		\item distance-hereditary graphs.
	\end{itemize}
\end{theorem}

\begin{proof}
	These are Corollaries \ref{Corollary_Cographs} and \ref{Corollary_PerfectGraphs} and Propositions \ref{Proposition_Chordal}, \ref{Proposition_Threshold}, \ref{Proposition_Split}, and \ref{Proposition_Distance}, respectively.
\end{proof}

\begin{center}
	\begin{tabular}{| l | l | l |}
		\hline \textbf{Class Name} & \textbf{Forbids} & \textbf{Indivisible} \\ \hline \hline
		Cographs & $P_4$ & Yes \\ \hline
		Perfect Graphs & $C_n$ and $\overline{C_n}$ for odd $n \geq 5$ & Yes \\ \hline
		Chordal Graphs & $C_n$ for all $n \geq 4$ & Yes \\ \hline
		Threshold Graphs & $P_4$, $C_4$, and $\overline{C_4}$ & No \\ \hline
		Split Graphs & $C_4$, $C_5$, and $\overline{C_4}$ & No \\ \hline
		Dist.-Her. Graphs & $\overline{P_5}$, $C_n$ for all $n \geq 5$, Domino, Gem & No \\ \hline
	\end{tabular}
\end{center}

Note that, although the class of distance-hereditary graphs is a subclass of the perfect graphs and a superclass of the cographs, both of which are indivisible, it is not.  Similarly, the class of cographs is indivisible despite being a subclass of the distance-hereditary graphs and a superclass of the threshold graphs, neither of which are indivisible.


\section{The Amalgamation Property}\label{Section_Amalgamation}

Indivisibility was originally examined for structures instead of classes (see, for example, \cite{EZS91}) and one convenient way of showing that a class of structures is indivisible is by showing that its generic limit is indivisible.
However, in the absence of the amalgamation property, no such generic limit exists.  This leads to the question: Which graph classes have the amalgamation property?  It turns out that very few graph classes do.

Throughout this section, let $\K$ be a class of finite graphs in the usual language of graphs that is closed under isomorphism and has the hereditary property.

The amalgamation property for classes of graphs was studied by A.~H.~Lachlan and R.~Woodrow in \cite{LW}.  In that paper, they prove the following theorem.

\begin{theorem}[Theorem 2' of \cite{LW}]\label{Theorem_LachlanWoodrow}
	If $\K$ has the amalgamation property, $P_3 \in \K$, $\overline{P_3} \in \K$, $N_m \in \K$ for all positive integers $m$, and $K_n \in \K$ for some positive integer $n$, then $\forbG(K_{n+1}) \subseteq \K$.
\end{theorem}

%
In \cite{LW}, they use Theorem \ref{Theorem_LachlanWoodrow} to obtain the following characterization of the amalgamation property for classes of graphs.

\begin{corollary}\label{Corollary_AmalgamationProperty}
	Suppose $\K$ is a class of finite graphs with the hereditary property and suppose that $\K$ contains graphs of arbitrarily large (finite) size.  Then, $\K$ has the amalgamation property if and only if $\K$ is equal to
	\begin{itemize}
		\item $\G$;
		\item $\forbG(K_n)$ for some $n \geq 2$;
		\item $\forbG(N_n)$ for some $n \geq 2$;
		\item $\forbG(P_3)$;
		\item $\forbG(\overline{P_3})$;
		\item $\forbG(P_3, K_n)$ for some $n \geq 3$;
		\item $\forbG(\overline{P_3}, K_n)$ for some $n \geq 3$;
		\item $\forbG(P_3, N_n)$ for some $n \geq 3$; or
		\item $\forbG(\overline{P_3}, N_n)$ for some $n \geq 3$.
	\end{itemize}
\end{corollary}

\begin{remark}
	As a result of Corollary \ref{Corollary_AmalgamationProperty}, most of the graph classes considered in this paper do not have the amalgamation property.  Certainly nothing from Section \ref{Section_Forbidden}.  On the other hand, by Proposition \ref{Proposition_classifyKlarge}, $\K_\alpha$ has the amalgamation property if and only if $\alpha > \frac{3}{2}$ (indeed, $\K_\alpha = \forbG(P_3, K_3)$ for $\alpha \in (\frac{3}{2},2]$ and $\K_\alpha = \forbG(K_2)$ for $\alpha > 2$).  Similarly, $\K^+_\alpha$ has the amalgamation property if and only if $\alpha \geq \frac{3}{2}$.
\end{remark}

Indivisibility for the classes of graphs with the amalgamation property is known or can be easily shown.

\begin{proposition}
	The following are indivisible:
	\begin{itemize}
		\item $\G$;
		\item $\forbG(K_n)$ for some $n \geq 2$;
		\item $\forbG(N_n)$ for some $n \geq 2$;
		\item $\forbG(P_3)$; and
		\item $\forbG(\overline{P_3})$.
	\end{itemize}
	Furthermore, the following are not indivisible:
	\begin{itemize}
		\item $\forbG(P_3, K_n)$ for some $n \geq 3$;
		\item $\forbG(\overline{P_3}, K_n)$ for some $n \geq 3$;
		\item $\forbG(P_3, N_n)$ for some $n \geq 3$; and
		\item $\forbG(\overline{P_3}, N_n)$ for some $n \geq 3$.
	\end{itemize}
\end{proposition}

\begin{proof}
	The indivisibility of $\G$ can be found, for example, in \cite{EZS89}.  It also follows from Theorem \ref{Theorem_Miriams}.  Indivisibility for $\forbG(K_n)$ for all $n \geq 2$ is in \cite{Folk70}.  By Proposition \ref{Proposition_PreserveComplement}, $\forbG(N_n)$ is also indivisible for all $n \geq 2$.  Indivisibility for $\forbG(P_3)$ follows from the Pigeonhole Principle and indivisibility of $\forbG(\overline{P_3})$ then follows from Proposition \ref{Proposition_PreserveComplement}.
	
	To see that $\forbG(P_3, K_n)$ is not indivisible for $n \geq 3$, take $A = K_{n-1}$, which is clearly in $\forbG(P_3, K_n)$.  For any $B \in \forbG(P_3, K_n)$, each component of $B$ must be a clique of size at most $n-1$.  For each component of size at least two, choose one vertex to color $1$ and color the rest $2$.  Color the single-vertex components arbitrarily.  Then, $B$ contains no monochromatic copy of $A$.  By Proposition \ref{Proposition_PreserveComplement}, $\forbG(\overline{P_3}, N_n)$ is also not indivisible for $n \geq 3$.
	
	To see that $\forbG(P_3, N_n)$ is not indivisible for $n \geq 3$, take $A = N_{n-1}$, which is clearly in $\forbG(P_3, N_n)$.  For any $B \in \forbG(P_3, N_n)$, $B$ must have at most $n-1$ components, each of which is a clique.  Color one component $1$ and the rest $2$.  Then, $B$ contains no monochromatic copy of $A$.  By Proposition \ref{Proposition_PreserveComplement}, $\forbG(\overline{P_3}, K_n)$ is also not indivisible for $n \geq 3$.
\end{proof}


\section{Future Directions}\label{Section_Conclusion}

We are interested in studying what other graph classes are indivisible.  For some graph classes, the answer is known, like planar graphs (see Example 2.10 of \cite{GPS}).  For others, it follows from Proposition \ref{Proposition_ForbidF}, as they are classified by forbidding sufficiently large paths, cycles, or their complements.  However, there are other classes of graphs for which indivisibility is unknown.

We are also interested in studying generalizations of indivisibility.  For example, there is a notion of ``strong substructure'' $\leq$ for classes of $L$-structures $\K$ that generalizes the notion of substructure in the usual (model-theoretic) sense (e.g., \cite{BS}).  Replacing ``substructure'' with ``strong substructure'' generalizes notions like the amalgamation property.     When $(\K, \leq)$ has the amalgamation property, it has a generic limit akin to the Fra\"{i}ss\'{e} limit \cite{BS}.  The classes $\K_\alpha$ and $\K^+_\alpha$ each come with a strong substructure notion, $\leq_\alpha$ and $\leq^+_\alpha$ respectively, such that $(\K_\alpha, \leq_\alpha)$ and $(\K^+_\alpha, \leq^+_\alpha)$ have the amalgamation property (e.g., \cite{BL}), hence these have generic limits.  If $\K$ is a class of $L$-structures and $\leq$ is a strong notion of substructure, we can say that $(\K, \leq)$ is \emph{indivisible} if, for all $A \in \K$ and positive integers $k$, there exists $B \in \K$ such that, for all $k$-colorings $c$ of $B$, there exists a monochromatic copy $f$ of $A$ in $B$ with respect to $c$ such that $f(A) \leq B$.  We can ask which pairs $(\K, \leq)$ are indivisible.  With respect to hereditarily $\alpha$-sparse graphs, it is clear that $(\K_\alpha, \leq_\alpha)$ is indivisible if and only if $\alpha > 2$ and $(\K^+_\alpha, \leq^+_\alpha)$ is indivisible if and only if $\alpha \geq 2$.  However, this is an interesting question for other pairs $(\K, \leq)$, including other graph classes.  Alternatively, for indivisibility of $(\K, \leq)$ we can color strong embeddings of singleton structures instead of elements of the structure.  What pairs $(\K, \leq)$ have this kind of indivisibility?

Another potential generalization of indivisibility would be to ask whether, instead of reducing to a single color, it is possible to reduce to a fixed finite number of colors, similar to the notion of small Ramsey degree.  Recall that, for $n \geq 3$, $\forbG(P_3, N_n)$ is not indivisible.  However, it is possible to reduce any finite coloring down to $n$ colors; i.e., for all $A \in \forbG(P_3, N_n)$, there exists $B \in \forbG(P_3, N_n)$ such that, for all $k > n$ and $k$-colorings $c$ of $B$, there exists an embedding $f : A \rightarrow B$ such that $|c(f(A))| \leq n$.  What other classes of structures (including graph classes) have this property?

\end{document}